\documentclass[11pt,reqno]{amsart}
\usepackage{times}

\newtheorem{thm}{Theorem}
\newtheorem*{thma}{Theorem A}
\newtheorem*{thmb}{Theorem B}
\newtheorem*{thmc}{Theorem C}

\newtheorem{lem}[thm]{Lemma}
\newtheorem{pro}[thm]{Proposition}

\newcommand{\C}{{\mathbb C}}
\newcommand{\inc}{\int_{\C}}

\newcommand{\re}{{\text{Re}}\,}

\begin{document}

\title{Fock-Sobolev spaces and their Carleson measures}

\author{Hong Rae Cho and Kehe Zhu}

\thanks{Cho was supported by the Korea Research Foundation Grant KRF-2009-0073976}

\address{Hong Rae Cho:
Department of Mathematics, Pusan
National University, Pusan 609-735, Korea}
\email{chohr@pusan.ac.kr}

\address{Kehe Zhu:
Department of Mathematics and Statistics,
State University of New York at Albany,
Albany, NY 12222, USA}
\email{kzhu@math.albany.edu}

\date\today
\subjclass[2000]{30H20}
\keywords{Fock space, Fock-Sobolev space, Carleson measure}

\begin{abstract}
We consider the Fock-Sobolev space $F^{p,m}$ consisting of entire functions $f$ such that
$f^{(m)}$, the $m$-th order derivative of $f$, is in the Fock space $F^p$. We show that an entire function $f$ is in
$F^{p,m}$ if and only if the function $z^mf(z)$ is in $F^p$. We also characterize the 
Carleson measures for the spaces $F^{p,m}$, establish the boundedness of the weighted 
Fock projection on appropriate $L^p$ spaces, identify the Banach dual of $F^{p,m}$,
and compute the complex interpolation space between two $F^{p,m}$ spaces.
\end{abstract}

\maketitle

\section{Introduction}

Let $\C$ be the complex plane and $dA$ be ordinary area measure on $\C$. For any
$0<p\le\infty$ let $F^p$ denote the space of entire functions $f$ such that the function
$f(z)e^{-|z|^2/2}$ is in $L^p(\C,dA)$. For $f\in F^p$ we write
$$\|f\|_p=\left[\frac p{2\pi}\inc\left|f(z)e^{-\frac12|z|^2}\right|^p\,dA(z)
\right]^{\frac 1p},\quad 0<p<\infty,$$
and
$$\|f\|_\infty=\sup_{z\in\C}|f(z)|e^{-\frac12|z|^2}.$$

The spaces $F^p$, especially $F^2$, have had a long history in mathematics and mathematical
physics and have been given a wide variety of appellations, including many combinations 
and permutations of the names Bargmann, Fischer, Fock, and Segal. See \cite{Ba,BC1,BC2,BC3,
Fi,Fo,JPR,Se1,Se2,Tu}. In this paper we are going to call them Fock spaces, for no particular
reason other than personal tradition. We refer the reader to \cite{Tu,Zhu} for more recent and
systematic treatment of Fock spaces.

To give a motivation for our study of Fock-Sobolev spaces, recall that the annihilation 
operator $A$ and the creation operator $A^*$ from the quantum theory of harmonic oscillators 
are defined by the commutation relation $[A,A^*]=I$, where $I$ is the identity operator. 
A natural representation of these operators is achieved on the Fock space $F^2$, namely,
$$Af(z)=f'(z),\qquad A^*f(z)=zf(z),\qquad f\in F^2.$$
It is easy to check that both $A$ and $A^*$, as defined above, are densely defined linear
operators on $F^2$ (unbounded though) and satisfy the commutation relation $[A,A^*]=I$.
Therefore, it is important to study the operator of multiplication by $z$ and the operator
of differentiation on the Fock space $F^2$. The commutation relation above also shows
that these two operators are intimately related. The purpose of this paper is to explore 
this relation a little bit further.

Thus for any positive integer $m$ we consider the space $F^{p,m}$ consisting of entire 
functions $f$ such that $f^{(m)}$, the $m$-th order derivative of $f$, belongs to $F^p$. 
Because of the similarity to the way the classical Sobolev spaces are defined, we are 
going to call $F^{p,m}$ Fock-Sobolev spaces. See \cite{HL,Th} for other similar 
Sobolev spaces.

The main results of the paper can be stated as follows.

\begin{thma}
Suppose $0<p\le\infty$, $m$ is a positive integer, and $f$ is an entire function on $\C$.
Then $f\in F^{p,m}$ if and only if the function $z^mf(z)$ is in $F^p$.
\end{thma}

\begin{thmb}
Suppose $0<p<\infty$, $m$ is a positive integer, and $r$ is a positive radius. Then
the following two conditions are equivalent for any positive Borel measure $\mu$ on $\C$.
\begin{enumerate}
\item[(a)] There exists a positive constant $C$ such that
$$\inc\left|f(z)e^{-\frac12|z|^2}\right|^p\,d\mu(z)\le C\|f\|^p_{p,m}$$
for all $f\in F^{p,m}$, where $\|f\|_{p,m}$ is the norm of the function $z^mf(z)$ in $F^p$.
\item[(b)] There exists a positive constant $C$ such that
$$\mu(B(z,r))\le C(1+|z|)^{mp}$$
for all $z\in\C$, where $B(z,r)=\{w\in\C:|w-z|<r\}$ is the Euclidean disk centered 
at $z$ with radius $r$.
\end{enumerate}
\end{thmb}

\begin{thmc}
Suppose $m$ is a positive integer and $L^p_m$ denote the space of Lebesgue measurable
functions $f$ on $\C$ such that $z^mf(z)e^{-|z|^2/2}$ is in $L^p(\C,dA)$. Then
\begin{enumerate}
\item[(a)] The orthogonal projection $Q_m:L^2_m\to F^{2,m}$ is a bounded projection
from $L^p_m$ onto $F^{p,m}$ when $1\le p\le\infty$.
\item[(b)] If $1\le p<\infty$ and $1/p+1/q=1$, then the Banach dual of $F^{p,m}$ can
be identified with $F^{q,m}$ under the integral pairing
$$\langle f,g\rangle=\inc f(z)\overline{g(z)}\,e^{-|z|^2}|z|^{2m}\,dA(z).$$
\item[(c)] If $0<p<1$, then the Banach dual of $F^{p,m}$ can be identified with
$F^{\infty,m}$ under the integral pairing above.
\item[(d)] If $1\le p_0<p_1\le\infty$ and $\theta\in(0,1)$, then
$$\left[F^{p_0,m},F^{p_1,m}\right]_\theta=F^{p,m},$$
where $p$ is determined by
$$\frac1p=\frac{1-\theta}{p_0}+\frac{\theta}{p_1}.$$
\end{enumerate}
\end{thmc}

Our results are based on the unweighted Fock space $F^p$. We could have started out with a
positive parameter $\lambda$ and considered the Fock spaces $F^p_\lambda$ consisting of 
entire functions $f$ such that the function $f(z)e^{-\lambda|z|^2/2}$ is in $L^p(\C,dA)$.
Everything we do in the paper can be generalized to this context. No additional ideas are
needed and no complications occur.

\section{Preliminary Estimates}

In this section we prove some preliminary estimates that will be needed in later
sections. We believe that these estimates are of some independent interest for future
research on Fock type spaces.

\begin{lem}
Suppose $s$ is real and $\sigma>0$. Then there exists a positive constant $C$ such that
$$\sum_{n=0}^\infty\left(\frac x{n+1}\right)^s\frac{x^n}{n!}\le Ce^x$$
for all $x\ge\sigma$. Furthermore, this holds for all $x\ge0$ if $s\ge0$.
\label{1}
\end{lem}

\begin{proof}
Let $[s]$ denote the largest integer less than or equal to $s$. Then $[s]\le s<[s]+1$.
For any $x>1$ let $N=N(x)$ denote the integer such that $N<x\le N+1$. Thus for any $x>1$
we have
\begin{eqnarray*}
\sum_{n=0}^\infty\left(\frac x{n+1}\right)^s\frac{x^n}{n!}&=&\sum_{n=0}^{N-1}\left(
\frac x{n+1}\right)^s\frac{x^n}{n!}+\sum_{n=N}^\infty\left(\frac x{n+1}\right)^s
\frac{x^n}{n!}\\
&\le&\sum_{n=0}^{N-1}\left(\frac x{n+1}\right)^{[s]+1}\frac{x^n}{n!}+
\sum_{n=N}^\infty\left(\frac x{n+1}\right)^{[s]}\frac{x^n}{n!}.
\end{eqnarray*}

If $[s]\ge0$ and $x>1$, it is clear that there exists a positive constant $C$ such that
$$\sum_{n=0}^\infty\left(\frac x{n+1}\right)^s\frac{x^n}{n!}\le C\left[\sum_{n=0}^{N-1}
\frac{x^{n+[s]+1}}{(n+[s]+1)!}+\sum_{n=N}^\infty\frac{x^{n+[s]}}{(n+[s])!}\right]
\le 2C\,e^x.$$
It is easy to see that the estimate above holds for $0\le x\le1$ as well if $s\ge0$, 
with a possibly different constant $C$ of course.

If $[s]<0$ and $x>1$, we can find positive constants $C_k$ such that
\begin{eqnarray*}
\sum_{n=0}^\infty\left(\frac x{n+1}\right)^s\frac{x^n}{n!}&\le&\sum_{n=0}^\infty\left(
\frac x{n+1}\right)^{[s]+1}\frac{x^n}{n!}+\sum_{n=0}^\infty\left(
\frac x{n+1}\right)^{[s]}\frac{x^n}{n!}\\
&\le&C_1+\!\!\sum_{n=-[s]-1}^\infty\left(\frac x{n+1}\right)^{[s]+1}\frac{x^n}{n!}
+\!\!\sum_{n=-[s]}^\infty\left(\frac x{n+1}\right)^{[s]}\frac{x^n}{n!}\\
&\le&C_1+C_2\left[\sum_{n=-[s]-1}^\infty\frac{x^{n+[s]+1}}{(n+[s]+1)!}+\sum_{n=-[s]}^\infty\frac{x^{n+[s]}}{(n+[s])!}\right]\\
&=&C_1+2C_2e^x\le C_3e^x.
\end{eqnarray*}
In the case $\sigma<1$, the desired estimate is obvious for $\sigma\le x\le1$ as both 
sides are strictly positive continuous functions.
\end{proof}

We show that the estimate in the lemma above can be reversed.

\begin{lem}
Suppose $s$ is real and $\sigma>0$. Then there exists a positive constant $C$ such that
$$\sum_{n=0}^\infty\left(\frac x{n+1}\right)^s\frac{x^n}{n!}\ge Ce^x$$
for all $x\ge\sigma$. Furthermore, the estimate above holds for all $x\ge0$ if $s\le0$.
\label{2}
\end{lem}

\begin{proof}
First assume that $0\le s\le1$. In this case, we start with $x>1$ and let $N=N(x)$ denote the 
positive integer such that $N<x\le N+1$. Then
$$\left(\frac x{n+1}\right)^s\ge1,\qquad 0\le n\le N-1,$$
and
$$\left(\frac x{n+1}\right)^s\ge\frac x{n+1},\qquad n\ge N.$$
It follows that
\begin{eqnarray*}
\sum_{n=0}^\infty\left(\frac x{n+1}\right)^s\frac{x^n}{n!}&=&\sum_{n=0}^{N-1}\left(
\frac x{n+1}\right)^s\frac{x^n}{n!}+\sum_{n=N}^\infty\left(\frac x{n+1}\right)^s
\frac{x^n}{n!}\\
&\ge&\sum_{n=0}^{N-1}\frac{x^n}{n!}+\sum_{n=N}^\infty\frac{x^{n+1}}{(n+1)!}\\
&=&e^x-\frac{x^N}{N!}=e^x\left(1-\frac{x^Ne^{-x}}{N!}\right).
\end{eqnarray*}
By Stirling's formula,
$$\lim_{x\to\infty}\left(1-\frac{x^Ne^{-x}}{N!}\right)=1.$$
It follows easily that there exists a positive constant $C$ such that
$$\sum_{n=0}^\infty\left(\frac x{n+1}\right)^s\frac{x^n}{n!}\ge Ce^x$$
for all $x\ge\sigma$.

If $s>1$, we have
\begin{eqnarray*}
\sum_{n=0}^\infty\left(\frac x{n+1}\right)^s\frac{x^n}{n!}&=&\sum_{n=0}^\infty
\left(\frac x{n+1}\right)^{s-[s]}\frac{x^{n+[s]}}{(n+1)^{[s]}n!}\\
&\ge&\sum_{n=0}^\infty\left(\frac x{n+1+[s]}\right)^{s-[s]}\frac{x^{n+[s]}}{(n+[s])!}\\
&=&\sum_{n=0}^\infty\left(\frac x{n+1}\right)^{s-[s]}\frac{x^n}{n!}-
\sum_{n=0}^{[s]-1}\left(\frac x{n+1}\right)^{s-[s]}\frac{x^n}{n!}.
\end{eqnarray*}
Since the second sum above grows polynomially and the first sum above grows at least 
like $e^x$ (by what was proved in the previous paragraph), we can find a positive constant 
$C$ such that
$$\sum_{n=0}^\infty\left(\frac x{n+1}\right)^s\frac{x^n}{n!}\ge Ce^x$$
for all $x\ge\sigma$.

If $s<0$, we can find positive constants $C_1$ and $C_2$ such that
\begin{eqnarray*}
\sum_{n=0}^\infty\left(\frac x{n+1}\right)^s\frac{x^n}{n!}&\ge&\sum_{n=-[s]}^\infty
\left(\frac x{n+1}\right)^s\frac{x^n}{n!}\\
&=&\sum_{n=-[s]}^\infty\left(\frac x{n+1}\right)^{s-[s]}\frac{x^{n+[s]}}{(n+1)^{[s]}n!}\\
&\ge& C_1\sum_{n=-[s]}^\infty\left(\frac x{n+1+[s]}\right)^{s-[s]}\frac{x^{n+[s]}}{(n+[s])!}\\
&=&C_1\sum_{n=0}^\infty\left(\frac x{n+1}\right)^{s-[s]}\frac{x^n}{n!}\\
&\ge& C_2e^x
\end{eqnarray*}
for all $x\ge\sigma$. The last estimate above is true because $0\le s-[s]<1$. The desired 
result is obvious for $x\in[0,\sigma]$ when $s\le0$.
\end{proof}

The following is a key estimate that will be used numerous times later on.

\begin{thm}
Suppose $m$ is a nonnegative integer and $p_m(z)$ is the Taylor polynomial of $e^z$ of
order $m-1$ (with the convention that $p_0=0$). For any parameters $p>0$,
$\sigma>0$, $a>0$, and $b>-(mp+2)$ we can find a positive constant $C$ such that
$$\inc|e^{z\overline w}-p_m(z\overline w)|^pe^{-a|w|^2}|w|^b\,dA(w)
\le C|z|^b e^{\frac{p^2}{4a}|z|^2}$$
for all $|z|\ge\sigma$. Furthermore, this holds for all $z$ if $b\le pm$ as well.
\label{3}
\end{thm}

\begin{proof}
Let $I(z)$ denote the integral in question. Then $I(z)=I_1(z)+I_2(z)$, where
$$I_1(z)=\int_{|w|\le1}|e^{z\overline w}-p_m(z\overline w)|^pe^{-a|w|^2}|w|^b\,dA(w),$$
and
$$I_2(z)=\int_{|w|\ge1}|e^{z\overline w}-p_m(z\overline w)|^pe^{-a|w|^2}|w|^b\,dA(w).$$
Now
\begin{eqnarray*}
I_1(z)&\le&\int_{|w|\le1}\left|\sum_{k=m}^\infty\frac{(z\overline w)^k}{k!}\right|^p
e^{-a|w|^2}|w|^b\,dA(w)\\
&=&|z|^{mp}\int_{|w|\le1}\left|\sum_{k=m}^\infty\frac{(z\overline w)^{k-m}}{k!}\right|^p
e^{-a|w|^2}|w|^{pm+b}\,dA(w)\\
&\le&\left[\sum_{k=m}^\infty\frac{|z|^k}{k!}\right]^p\int_{|w|\le1}e^{-a|w|^2}|w|^{pm+b}
\,dA(w)\\
&=&C[e^{|z|}-p_m(|z|)]^p.
\end{eqnarray*}
Note that integration in polor coordinates shows that the last integral above 
converges at the origin whenever $pm+b>-2$. 

On the other hand,
$$I_2(z)\le 2^p(J_1(z)+J_2(z)),$$
where
\begin{eqnarray*}
J_1(z)&=&\int_{|w|\ge1}|p_m(z\overline w)|^pe^{-a|w|^2}|w|^b\,dA(w)\\
&\le&\int_{|w|\ge1}\left[\sum_{k=0}^{m-1}\frac{|z|^k|w|^k}{k!}\right]^pe^{-a|w|^2}|w|^b\,dA(w)\\
&\le&\left[\sum_{k=0}^{m-1}\frac{|z|^k}{k!}\right]^p\int_{|w|\ge1}
e^{-a|w|^2}|w|^{p(m-1)+b}\,dA(w)\\
&=&C(p_m(|z|))^p,
\end{eqnarray*}
and
\begin{eqnarray*}
J_2(z)&=&\int_{|w|\ge1}|e^{z\overline w}|^pe^{-a|w|^2}|w|^b\,dA(w)\\
&=&\int_{|w|\ge1}|e^{pz\overline w/2}|^2e^{-a|w|^2}|w|^b\,dA(w)\\
&=&\int_{|w|\ge1}\left|\sum_{n=0}^\infty\frac{(pz\overline w/2)^n}{n!}\right|^2
e^{-a|w|^2}|w|^b\,dA(w)\\
&=&\sum_{n=0}^\infty\frac{(p|z|/2)^{2n}}{(n!)^2}\int_{|w|\ge1}|w|^{2n+b}e^{-a|w|^2}\,dA(w)\\
&=&\pi\sum_{n=0}^\infty\frac{(p|z|/2)^{2n}}{(n!)}\int_1^\infty r^{n+(b/2)}e^{-ar}\,dr\\
&=&\frac{\pi}{a^{1+(b/2)}}\sum_{n=0}^\infty\frac{(p^2|z|^2/4|a|)^n}{(n!)^2}\int_a^\infty
r^{n+(b/2)}e^{-r}\,dr.
\end{eqnarray*}
We can find a constant $C_1>0$ such that
$$\frac1{n!}\int_a^\infty r^{n+(b/2)}e^{-r}\,dr\le C_1(n+1)^{b/2}$$
for the possible few $n$ that satisfies $n+(b/2)<0$. For $n+(b/2)\ge0$ we use Stirling's 
formula to find a positive constant $C_2$ (independent of $n$) such that
\begin{eqnarray*}
\frac1{n!}\int_a^\infty r^{n+(b/2)}e^{-r}\,dr&\le&\frac1{n!}\int_0^\infty r^{n+(b/2)}
e^{-r}\,dr\\
&=&\frac{\Gamma(n+(b/2)+1)}{\Gamma(n+1)}\\
&\le& C_2(n+1)^{b/2}.
\end{eqnarray*}
Therefore, there exists a positive constant $C$ such that
$$J_2(z)\le C|z|^b\sum_{n=0}^\infty\frac{(p^2|z|^2/4a)^n}{n!}\left(
\frac{n+1}{p^2|z|^2/4a}\right)^{b/2}.$$
This along with Lemma~\ref{1} shows that we find another positive constant $C$ such that
$$J_2(z)\le C|z|^be^{\frac{p^2}{4a}|z|^2},\qquad |z|\ge\sigma.$$
Combining the estimates for $I_1(z)$, $J_1(z)$, and $J_2(z)$, we find another positive
constant $C$ such that 
$$I(z)\le C|z|^be^{\frac{p^2}{4a}|z|^2},\qquad |z|\ge\sigma.$$

Finally, if $b\le pm$ as well and $|z|\le\sigma$, we have
\begin{eqnarray*}
I(z)&=&\inc|e^{z\overline w}-p_m(z\overline w)|^pe^{-a|w|^2}|w|^b\,dA(w)\\
&=&\inc\left|\sum_{n=m}^\infty\frac{[(z/\sigma)(\sigma\overline w)]^n}{n!}\right|^p
e^{-a|w|^2}|w|^b\,dA(w)\\
&\le&\left|\frac z\sigma\right|^{pm}\inc\left|\sum_{n=m}^\infty\frac{(\sigma|w|)^n}
{n!}\right|^pe^{-a|w|^2}|w|^b\,dA(w)\\
&=&\left|\frac z\sigma\right|^b\inc\left[e^{\sigma|w|}-p_m(\sigma|w|)\right]^p
e^{-a|w|^2}|w|^b\,dA(w)\\
&\le& C|z|^be^{\frac{p^2}{4a}|z|^2}.
\end{eqnarray*}
This completes the proof of the theorem.
\end{proof}

\section{The Fock-Sobolev spaces $F^{p,m}$}

The purpose of this section is to show that, for an entire function $f$ and a positive 
integer $m$, the function $f^{(m)}(z)$ belongs to $F^p$ if and only if the function 
$z^mf(z)$ is in $F^p$.

For any $0<p\le\infty$ let $L^p_*$ denote the space of measurable functions $f$ such that
the function $f(z)e^{-|z|^2/2}$ is in $L^p(\C,dA)$. If $0<p<\infty$, the norm in $L^p_*$ is 
defined by
$$\|f\|_p^p=\frac{p}{2\pi}\inc\left|f(z)e^{-\frac12|z|^2}\right|^p\,dA(z).$$
For $p=\infty$, the norm in $L^\infty_*$ is defined by
$$\|f\|_\infty=\sup_{z\in\C}|f(z)|e^{-\frac12|z|^2}.$$

Recall that $F^2$ is a closed subspace of $L^2_*$, and the orthogonal projection
$P:L^2_*\to F^2$ is given by
$$Pf(z)=\frac1\pi\inc e^{z\overline w}f(w)e^{-|w|^2}\,dA(w).$$
It is well known that for $1\le p\le\infty$ the Fock space $F^p$ is a closed subspace
of $L^p_*$ and the Fock projection $P$ above is a bounded projection from $L^p_*$ 
onto $F^p$. See \cite{JPR} for example.

\begin{lem}
Suppose $\lambda>0$ and $0<p\le1$. There exists a positive constant $C$ such that
$$\inc|f(z)|e^{-\frac\lambda2|z|^2}\,dA(z)\le C\left[\inc\left|f(z)
e^{-\frac\lambda2|z|^2}\right|^p\,dA(z)\right]^{1/p}$$
for all entire functions $f$.
\label{4}
\end{lem}

\begin{proof}
It is well known that $|f(z)|\le\|f\|_pe^{|z|^2/2}$ for all entire functions $f$
and all $z\in\C$. See \cite{JPR} or \cite{Tu}. By a simple of change of variables, this 
can easily be generalized to weighted Fock spaces with the Gaussian weight 
$e^{-\lambda|z|^2}$. More specifically, we have 
\begin{equation}
|f(z)|e^{-\lambda|z|^2/2}\le\left[\frac{p\lambda}{2\pi}\inc\left|f(z)
e^{-\frac\lambda2|z|^2}\right|^p\,dA(z)\right]^{\frac1p}
\label{eq-add}
\end{equation}
for all entire functions $f$ and all $z\in\C$. See \cite{Tu} and \cite{Zhu} for example.

The desired inequality now follows easily by writing 
$$|f(z)|e^{-\frac\lambda2|z|^2}=\left|f(z)e^{-\lambda|z|^2/2}\right|^p
\left|f(z)e^{-\lambda|z|^2/2}\right|^{1-p}$$
and estimating the factor $|f(z)e^{-\lambda|z|^2/2}|^{1-p}$ using (\ref{eq-add}).
\end{proof}

As a consequence of the lemma above, we have $F^p\subset F^1$ when $0<p\le1$. More
generally, we always have $F^p\subset F^q$ whenever $0<p\le q\le\infty$. See \cite{JPR}
and \cite{Tu} for example.

\begin{lem}
Suppose $0<p\le\infty$, $m$ is a positive integer, and $f$ is an entire function.
If the function $z^mf(z)$ is in $F^p$, then so is $f^{(m)}$.
\label{5}
\end{lem}

\begin{proof}
If the function $w^mf(w)$ is in $F^p$, then the following reproducing formula holds:
$$f(z)=\frac1\pi\inc e^{z\overline w}f(w)e^{-|w|^2}\,dA(w),\qquad z\in\C.$$
The convergence of the integral above follows from (\ref{eq-add}).
Differentiating under the integral sign $m$ times, we obtain
\begin{equation}
f^{(m)}(z)=\frac1\pi\inc e^{z\overline w}\,\overline w^mf(w)e^{-|w|^2}\,dA(w).
\label{eq1}
\end{equation}
In other words, we have $f^{(m)}=Pg$, where $g(z)=\overline z^m\,f(z)$. Since $g\in L^p_*$ 
and $P$ is a projection of $L^p_*$ onto $F^p$ for $1\le p\le\infty$, the desired result
is clear when $1\le p\le\infty$.

When $0<p<1$, we note from (\ref{eq1}) that
$$|f^{(m)}(z)|\le\frac1\pi\inc|w^mf(w)e^{\overline zw}|e^{-|w|^2}\,dA(w).$$
By Lemma~\ref{4}, there exists a positive constant $C$, independent of $f$, such that
$$|f^{(m)}(z)|^p\le C\inc\left|w^mf(w)e^{\overline zw}e^{-|w|^2}\right|^p\,dA(w).$$
It follows from this and Fubini's theorem that we can estimate the integral
$$I=\inc|f^{(m)}(z)e^{-|z|^2/2}|^p\,dA(z)$$
as follows.
\begin{eqnarray*}
I&\le&C\inc e^{-p|z|^2/2}\,dA(z)\inc|w^mf(w)|^pe^{-p|w|^2}|e^{\overline zw}|^p\,dA(w)\\
&=&C\inc|w^mf(w)|^pe^{-p|w|^2}\,dA(w)\inc|e^{\frac p2w\overline z}|^2e^{-\frac p2|z|^2}\,dA(z)\\
&=&\frac{2\pi C}p\inc|w^mf(w)|^pe^{-p|w|^2}e^{\frac p2|w|^2}\,dA(w)\\
&=&\frac{2\pi C}p\inc|w^mf(w)e^{-|w|^2/2}|^p\,dA(w).
\end{eqnarray*}
This shows that $f^{(m)}\in F^p$ and completes the proof of the lemma.
\end{proof}

\begin{lem}
Suppose $0<p\le\infty$, $m$ is a positive integer, and $f$ is an entire function. If 
the function $f^{(m)}(z)$ is in $F^p$, then so is the function $z^mf(z)$.
\label{6}
\end{lem}

\begin{proof}
Since $f^{(m)}$ is in $F^p$, we have the reproducing formula
$$f^{(m)}(z)=\frac1\pi\inc e^{z\overline w}f^{(m)}(w)e^{-|w|^2}\,dA(w),\qquad z\in\C.$$
Again, the convergence of the integral above follows from (\ref{eq-add}). Integrating from 
$0$ to $z$ repeatedly for $m$ times, we obtain
\begin{equation}
f(z)-f_m(z)=\frac1\pi\inc\frac{e^{z\overline w}-p_m(z\overline w)}{\overline w^m}
\,f^{(m)}(w)e^{-|w|^2}\,dA(w),
\label{eq2}
\end{equation}
where
$$f_m(z)=\sum_{k=0}^{m-1}\frac{f^{(k)}(0)}{k!}z^k,\quad p_m(z)=\sum_{k=0}^{m-1}
\frac{z^k}{k!}.$$

We first consider the case $p=\infty$. In this case, it follows from (\ref{eq2})
and Theorem~\ref{3} that there is a positive constant $C$ such that
\begin{eqnarray*}|f(z)-f_m(z)|&\le&\frac1\pi\|f^{(m)}\|_{\infty}\inc|e^{z\overline w}
-p_m(z\overline w)|\,e^{-|w|^2/2}|w|^{-m}\,dA(w)\\
&\le& C\|f^{(m)}\|_{\infty}|z|^{-m}e^{|z|^2/2}
\end{eqnarray*}
for all $z$. Since every polynomial belongs to $F^\infty$, we deduce from this that the 
function $z^mf(z)$ is in $F^\infty$ whenever $f^{(m)}(z)$ is in $F^\infty$.

Next we consider the case $p=1$. It follows from (\ref{eq2}) and Fubini's theorem that
\begin{eqnarray*}
&&\inc|z|^m|f(z)-f_m(z)|e^{-|z|^2/2}\,dA(z)\\
&\le&\frac1\pi\inc|z|^me^{-|z|^2/2}\,dA(z)\inc|e^{z\overline w}-p_m(z\overline w)|
|f^{(m)}(w)|e^{-|w|^2}|w|^{-m}\,dA(w)\\
&=&\frac1\pi\inc e^{-|w|^2}|f^{(m)}(w)||w|^{-m}\,dA(w)\inc|e^{z\overline w}-
p_m(z\overline w)|e^{-|z|^2/2}|z|^m\,dA(z).
\end{eqnarray*}
This together with Theorem~\ref{3} shows that there is a positive 
constant $C$ such that
$$\inc|z|^m|f(z)-f_m(z)|e^{-|z|^2/2}\,dA(z)\le C\inc|f^{(m)}(w)|e^{-|w|^2/2}\,dA(w).$$
Thus $z^mf(z)$ is in $F^1$ whenever $f^{(m)}(z)$ is in $F^1$.

With the help of complex interpolation we conclude that the function $z^mf(z)$ is in 
$F^p$ whenever the function $f^{(m)}(z)$ is in $F^p$, where $1\le p\le\infty$.

Finally, if $0<p<1$, we note from (\ref{eq2}) that
$$|f(z)-f_m(z)|\le\frac1\pi\inc\left|\frac{e^{\overline zw}-p_m(\overline zw)}{w^m}
f^{(m)}(w)e^{-|w|^2}\right|\,dA(w).$$
Again, by Lemma~\ref{4} (with $\lambda=2$), there exists a positive constant $C$ 
such that
$$|f(z)-f_m(z)|^p\le C\inc\left|\frac{e^{\overline zw}-p_m(\overline zw)}{w^m}f^{(m)}(w)
e^{-|w|^2}\right|^p\,dA(w).$$
This together with Fubini's theorem shows that the integral
$$J=\inc|z^m(f(z)-f_m(z))e^{-|z|^2/2}|^p\,dA(z)$$
satisfies the following estimates:
\begin{eqnarray*}
J&\le&C\inc|z|^{mp}e^{-\frac p2|z|^2}\,dA(z)\inc|e^{z\overline w}-p_m(z\overline w)|^p
|f^{(m)}(w)|^pe^{-p|w|^2}\frac{dA(w)}{|w|^{mp}}\\
&=&C\inc|f^{(m)}(w)|^pe^{-p|w|^2}\frac{dA(w)}{|w|^{mp}}\inc|e^{z\overline w}-
p_m(z\overline w)|^pe^{-\frac p2|z|^2}|z|^{mp}\,dA(z).
\end{eqnarray*}
By Theorem~\ref{3} there exists another positive constant $C$ such that
$$\inc|e^{z\overline w}-p_m(z\overline w)|^pe^{-\frac p2|z|^2}|z|^{mp}\,dA(z)
\le C|w|^{mp}e^{\frac p2|w|^2}$$
for all $w\in\C$. This shows that
$$J\le C\inc|f^{(m)}(w)e^{-\frac12|w|^2}|^p\,dA(w)$$
for another positive constant $C$ that is independent of $f$. Therefore, the function 
$z^m(f(z)-f_m(z))$ is in $F^p$. Since every polynomial is in $F^p$, we conclude that 
the function $z^mf(z)$ is in $F^p$. This completes the proof of the lemma.
\end{proof}

Combining Lemmas \ref{5} and \ref{6}, we have now proved the following theorem, the
main result of the section.

\begin{thm}
Suppose $0<p\le\infty$, $m$ is a positive integer, and $f$ is an entire function.
Then $f^{(m)}$ is in $F^p$ if and only if the function $z^mf(z)$ is in $F^p$.
\label{7}
\end{thm}

\section{Carleson Measures}

Since $F^{p,m}$ consists of entire functions $f$ such that $z^mf(z)$ is in $F^p$, it
is natural to consider the following norm on $F^{p,m}$ when $0<p<\infty$:
\begin{equation}
\|f\|^p_{p,m}=c(p,m)\inc|z^mf(z)e^{-|z|^2/2}|^p\,dA(w),
\label{eq3}
\end{equation}
where
$$c(p,m)=\frac{(p/2)^{(mp/2)+1}}{\pi\,\Gamma((mp/2)+1)}$$
is a normalizing constant so that the constant function $1$ has norm $1$ in $F^{p,m}$.
By the open mapping theorem, we have
$$\|f\|_{p,m}\sim|f(0)|+\cdots+|f^{(m-1)}(0)|+\|f^{(m)}\|_{F^p}.$$
It is more convenient for us to use the norm $\|f\|_{p,m}$ defined in (\ref{eq3}).

\begin{pro}
The Hilbert space $F^{2,m}$ possesses the following orthonormal basis:
$$e_n(z)=\sqrt{\frac{\Gamma(m+1)}{\Gamma(n+m+1)}}\,z^n=\sqrt{\frac{m!}{(n+m)!}}\,z^n,
\qquad n=0,1,2,\cdots.$$
Consequently, the reproducing kernel of $F^{2,m}$ is given by
$$K_m(z,w)=\sum_{n=0}^\infty\frac{m!}{(n+m)!}(z\overline w)^n=
\frac{m!}{(z\overline w)^m}\left[e^{z\overline w}-p_m(z\overline w)\right],$$
where $p_m$ is the Taylor polynomial of $e^z$ of order $m-1$.
\label{8}
\end{pro}

\begin{proof}
The first assertion follows from elementary calculations using polar coordinates. 
The second assertion follows from the fact that
$$K_m(z,w)=\sum_{n=0}^\infty e_n(z)\overline{e_n(w)},$$
where $\{e_n\}$ is any orthonormal basis of $F^{2,m}$.
\end{proof}

\begin{lem}
Suppose $p>0$, $t>0$, and $\lambda>0$. There exists a positive constant $C$ such that
$$\left|f(z)e^{-\frac\lambda2|z|^2}\right|^p\le C\int_{|w-z|<t}\left|f(w)
e^{-\frac\lambda2|w|^2}\right|^p\,dA(w)$$
for all entire functions $f$ and all $z\in\C$.
\label{9}
\end{lem}

\begin{proof}
This is well known. See \cite{IZ} or \cite{Zhu}.
\end{proof}

We now prove the main result of this section.

\begin{thm}
Suppose $0<p<\infty$, $m$ is a positive integer, $r$ is a positive radius, and
$\mu$ is a positive Borel measure on $\C$. Then the following two conditions are 
equivalent.
\begin{enumerate}
\item[(a)] There exists a positive constant $C$ such that
\begin{equation}
\inc\left|f(z)e^{-\frac12|z|^2}\right|^p\,d\mu(z)\le C\|f\|^p_{p,m}
\label{eq4}
\end{equation}
for all $f\in F^{p,m}$.
\item[(b)] There exists a positive constant $C$ such that
\begin{equation}
\mu(B(a,r))\le C(1+|a|)^{mp}
\label{eq5}
\end{equation}
for all $a\in\C$, where $B(a,r)=\{z\in\C:|z-a|<r\}$ is the Euclidean disk centered 
at $a$ with radius $r$.
\end{enumerate}
\label{10}
\end{thm}

\begin{proof}
First assume that there is a positive constant $C$ such that (\ref{eq4}) holds for
all $f\in F^{p,m}$. Taking $f=1$ shows that $\mu(S)<\infty$ for any compact set $S$.

Fix any complex number $a$ and let
$$f(z)=[e^{z\overline a}-p_m(z\overline a)]/z^m$$
in (\ref{eq4}). Then it follows from Theorem~\ref{3} that there exists another
constant $C>0$, independent of $a$, such that
$$\inc\left|\frac{e^{z\overline a}-p_m(z\overline a)}{z^m}\,e^{-\frac12|z|^2}\right|^p
\,d\mu(z)\le Ce^{\frac p2|a|^2}.$$
In particular,
$$\int_{|z-a|<r}\left|\frac{e^{z\overline a}-p_m(z\overline a)}{z^m}\,e^{-\frac12|z|^2}
\right|^p\,d\mu(z)\le Ce^{\frac p2|a|^2}.$$
If $|a|>2r$, then $|z|^m$ is comparable to $(1+|a|)^m$ for $|z-a|<r$. So there is another
positive constant $C$ such that
$$\int_{|z-a|<r}|e^{z\overline a}|^p|1-e^{-z\overline a}p_m(z\overline a)|^p
e^{-\frac p2|z|^2}\,d\mu(z)\le C(1+|a|)^{mp}e^{\frac p2|a|^2}$$
for all $|a|>2r$. It is easy to show that
$$\lim_{a\to\infty,|z-a|<r}(1-e^{-z\overline a}p_m(z\overline a))=1.$$
Thus we can find another constant $C>0$ with
$$\int_{|z-a|<r}|e^{z\overline a}|^pe^{-\frac p2|z|^2}\,d\mu(z)\le C(1+|a|)^{mp}
e^{\frac p2|a|^2}$$
for large $|a|$. But this is clearly true (with a different constant $C$) for smaller $|a|$
as well. So we can find another constant $C>0$ such that
$$\int_{|z-a|<r}e^{-\frac p2|a|^2}|e^{z\overline a}|^pe^{-\frac p2|z|^2}\,d\mu(z)
\le C(1+|a|)^{mp}$$
for all $a\in\C$. Completing a square in the exponent, we can rewrite the inequality above
as
$$\int_{|z-a|<r}e^{-\frac p2|z-a|^2}\,d\mu(z)\le C(1+|a|)^{mp},$$
from which we deduce that
$$\mu(B(a,r))\le Ce^{-\frac p2r^2}(1+|a|)^{mp}$$
for all $a\in\C$. This shows that condition (a) implies condition (b).

Next we assume that there exists a constant $C>0$ such that (\ref{eq5}) holds for all
$a\in\C$. We proceed to estimate the integral
$$I(f)=\inc\left|f(z)e^{-\frac12|z|^2}\right|^p\,d\mu(z)$$
for any function $f\in F^{p,m}$.

For any positive $s$ let $Q_s$ denote the following square in $\C$ with vertices $0$, $s$, 
$si$, and $s+si$:
$$Q_s=\{z=x+iy:0<x\le s,0<y\le s\}.$$
Let $Z$ denote the set of all integers and consider the lattice
$$Z_s^2=\{ns+ims:n\in Z,m\in Z\}.$$
It is clear that
$$\C=\bigcup\{Q_s+a:a\in Z^2_s\}$$
is decomposition of $\C$ into disjoint squares of side length $s$. Thus
$$I(f)=\sum_{a\in Z^2_s}\int_{Q_s+a}\left|f(z)e^{-\frac12|z|^2}\right|^p\,d\mu(z).$$

We fix positive numbers $s$ and $t$ such that $t+\sqrt s=r$. By Lemma~\ref{9} there
exists a constant $C$ such that
$$|f(z)|^pe^{-\frac p2|z|^2}\le C\int_{|w-z|<t}|f(w)e^{-\frac12|w|^2}|^p\,dA(w)$$
for all $z\in\C$. From this we easily deduce that
$$|f(z)e^{-\frac12|z|^2}|^p\le\frac C{(1+|z|)^{mp}}\int_{|w-z|<t}|w^mf(w)
e^{-\frac12|w|^2}|^p\,dA(w)$$
for all $z\in\C$, where $C$ is another positive constant. Now if $z\in Q_s+a$, where
$a\in Z^2_s$, then $B(z,t)\subset B(a,r)$ by the triangle inequality, and $1+|z|$ is
comparable to $1+|a|$. It follows that
$$|f(z)e^{-\frac12|z|^2}|^p\le\frac C{(1+|a|)^{mp}}\int_{|w-a|<r}|w^mf(w)
e^{-\frac12|w|^2}|^p\,dA(w),$$
where $C$ is another positive constant. Therefore,
$$I(f)\le C\sum_{a\in Z^2_s}\frac{\mu(B(a,r))}{(1+|a|)^{mp}}\int_{B(a,r)}
\left|w^mf(w)e^{-\frac12|w|^2}\right|^p\,dA(w).$$
Combining this with the assumption in (\ref{eq5}), we find another positive constant $C$
such that
$$I(f)\le C\sum_{a\in Z^2_s}\int_{B(a,r)}\left|w^mf(w)e^{-\frac12|w|^2}\right|^p
\,dA(w).$$
It is clear that there exists a positive integer $N$ such that each point in the complex
plane belongs to at most $N$ of the disks $B(a,r)$, where $a\in Z^2_s$. It follows that
$$I(f)\le CN\inc\left|w^mf(w)e^{-\frac12|w|^2}\right|^p\,dA(w).$$
Since $C$ and $N$ are independent of $f$, this shows that condition (b) implies 
condition (a). 
\end{proof}

The following is the ``little oh'' version of Theorem~\ref{10}. The proof is similar
to that of Theorem~\ref{10} and we omit the details here.

\begin{thm}
Suppose $0<p<\infty$, $m$ is a positive integer, $r$ is a positive radius, and
$\mu$ is a positive Borel measure on $\C$. Then
$$\lim_{a\to\infty}\frac{\mu(B(a,r))}{(1+|a|)^{mp}}=0$$
if and only if $F^{p,m}\subset L^p(\C,d\mu)$ and the inclusion is compact.
\label{11}
\end{thm}

Note that compact inclusion in the theorem above means the following: whenever $\{f_n\}$
is a bounded sequence in $F^{p,m}$ that converges to $0$ uniformly on compact sets we have
$$\lim_{n\to\infty}\inc\left|f_n(z)e^{-\frac12|z|^2}\right|^p\,d\mu(z)=0.$$
When $1<p<\infty$, this is consistent with the definition of compact linear operators in
the theory of Banach spaces.

\section{Duality and complex interpolation}

Recall that the Fock-Sobolev space $F^{p,m}$ consists of entire functions $f$ such that 
the function $z^mf(z)e^{-|z|^2/2}$ is in $L^p(\C,dA)$. More generally, we let $L^p_m$ denote 
the space of Lebesgue measurable functions $f$ on the complex plane such that the function
$z^mf(z)e^{-|z|^2/2}$ is in $L^p(\C,dA)$. When $0<p<\infty$, we use the following
norm in $L^p_m$:
$$\|f\|^p_{p,m}=c(m,p)\inc\left|z^mf(z)e^{-|z|^2/2}\right|^p\,dA(z),$$
where $c(m,p)$ is a normalizing constant defined earlier. When $p=\infty$, we use the 
following norm in $L^\infty_m$:
$$\|f\|_{\infty,m}=\sup\{|z^mf(z)e^{-|z|^2/2}|:z\in\C\}.$$

The space $L^2_m$ is a Hilbert space and the inner product on it is given by
\begin{equation}
\langle f,g\rangle_m=\frac1{m!\pi}\inc f(z)\overline{g(z)}e^{-|z|^2}|z|^{2m}\,dA(z).
\label{eq6}
\end{equation}
The Fock-Sobolev space $F^{2,m}$ is a closed subspace of $L^2_m$, and the orthogonal
projection $Q_m:L^2_m\to F^{2,m}$ is given by
\begin{equation}
Q_mf(z)=\frac1{m!\pi}\inc f(w)K_m(z,w)e^{-|w|^2}|w|^{2m}\,dA(w),
\label{eq7}
\end{equation}
where $K_m(z,w)$ is the reproducing kernel of $F^{2,m}$ given in Proposition~\ref{8}.

\begin{thm}
For each $1\le p\le\infty$ the integral operator $Q_m$ in (\ref{eq7}) is a bounded
projection from $L^p_m$ onto $F^{p,m}$.
\label{12}
\end{thm}

\begin{proof}
It suffices to show that the linear operator $Q_m$ is bounded on $L^p_m$ for 
$1\le p\le\infty$. The reproducing property of $K_m$ then shows that $Q_m$ is a projection
from $L^p_m$ onto $F^{p,m}$.

We first consider the case when $p=\infty$. If $f\in L^\infty_m$, then
\begin{eqnarray*}
|Q_mf(z)|&\le&\frac1{m!\pi}\inc|f(w)|\frac{|e^{z\overline w}-p_m(z\overline w)|}{|zw|^m}
e^{-|w|^2}|w|^{2m}\,dA(w)\\
&\le&\frac1{m!\pi|z|^m}\|f\|_{\infty,m}\inc|e^{z\overline w}-p_m(z\overline w)|
e^{-|w|^2/2}\,dA(w).
\end{eqnarray*}
By Theorem \ref{3}, there exists a constant $C>0$, independent of $z$ and $f$, such that
$$|z^mQ_mf(z)|\le C\|f\|_{\infty,m}e^{|z|^2/2},\qquad z\in\C.$$
This shows that $\|Q_mf\|_{\infty,m}\le C\|f\|_{\infty,m}$ for all $f\in L^\infty_m$.
Thus $Q_m$ is bounded on $L^\infty_m$.

We next consider the case when $p=1$. So let $f\in L^1_m$ and consider the integral
$$I=m!\pi\inc|z^mQ_mf(z)e^{-|z|^2/2}|\,dA(z).$$
By Fubini's theorem,
\begin{eqnarray*}
I&\le&\inc|z|^me^{-|z|^2/2}\,dA(z)\inc\frac{|f(w)||e^{z\overline w}-
p_m(z\overline w)|}{|zw|^m}e^{-|w|^2}|w|^{2m}\,dA(w)\\
&=&\inc|w^mf(w)|e^{-|w|^2}\,dA(w)\inc|e^{z\overline w}-
p_m(z\overline w)|e^{-|z|^2/2}\,dA(z).
\end{eqnarray*}
According to Theorem \ref{3} again, there exists a positive constant $C$, independent 
of $f$, such that
$$I\le C\inc|w^mf(w)|e^{-|w|^2/2}\,dA(w).$$
This shows that $Q_m$ is bounded on $L^1_m$.

It is well known that the scale of spaces $L^p_m$ interpolate in the usual way:
$$[L^{p_0}_m,L^{p_1}_m]_\theta=L^p_m,$$
where $1\le p_0<p_1\le\infty$, $\theta\in(0,1)$, and
$$\frac1p=\frac{1-\theta}{p_0}+\frac{\theta}{p_1}.$$
See \cite{SW} for example. Therefore, the boundedness of the linear operator $Q_m$ 
on $L^\infty_m$ and $L^1_m$ imply the boundedness of $Q_m$ on $L^p_m$ for 
all $1\le p\le\infty$.
\end{proof}

\begin{thm}
Suppose $1\le p<\infty$ and $1/p+1/q=1$. Then $(F^{p,m})^*$, the Banach dual of $F^{p,m}$,
can be identified with $F^{q,m}$ under the integral pairing in (\ref{eq6}).
\label{13}
\end{thm}

\begin{proof}
That every function in $F^{q,m}$ induces a bounded linear functional on $F^{p,m}$ via
the integral pairing in (\ref{eq6}) follows from the Cauchy-Schwarz inequality.

On the other hand, if $F$ is a bounded linear functional on $F^{p,m}$, then according to
the Hahn-Banach extension theorem, $F$ can be extended (without increasing its norm)
to a bounded linear functional on $L_m^p$. By the usual duality of $L_m^p$ spaces,
there exists some $h\in L_m^q$ such that
$$F(f)=\langle f, h\rangle_m,\qquad f\in F^{p,m}.$$
By Theorem \ref{12}, $Q_m$ is a bounded projection from $L^p_m$ onto $F^{p,m}$.
So if we let $g=Q_m(h)$, then $g\in F^{p,m}$ and
$$F(f)=\langle f, h\rangle_m=\langle Q_m(f), h\rangle_m
=\langle f, Q_m(h)\rangle_m=\langle f, g\rangle_m$$
for all $f\in F^{p,m}$. This completes the proof of the theorem.
\end{proof}

A similar result holds for $0<p<1$.

\begin{thm}
Suppose $0<p<1$ and $m$ is a positive integer. Then the dual space of $F^{p,m}$ can be
identified with $F^{\infty,m}$ under the integral pairing in (\ref{eq6}).
\label{14}
\end{thm}

\begin{proof}
First suppose that $g\in F^{\infty,m}$ and
$$F(f)=\inc f(z)\overline{g(z)}\,e^{-|z|^2}|z|^{2m}\,dA(z),\qquad f\in F^{p,m}.$$
Then
$$|F(f)|\le\|g\|_{\infty,m}\inc|z^mf(z)e^{-|z|^2/2}|\,dA(z),$$
and an application of Lemma~\ref{4} to the function $z^mf(z)$ yields
$$|F(f)|\le C\|g\|_{\infty,m}\|f\|_{p,m},\qquad f\in F^{p,m},$$
where $C$ is a positive constant independent of $f$. This shows that $F$ defines
a bounded linear functional on $F^{p,m}$.

Conversely, if $F$ is a bounded linear functional on $F^{p,m}$, then for any
$f\in F^{p,m}$ we deduce from the reproducing formula
$$f(z)=\frac1{m!\pi}\inc f(w)K_m(z,w)e^{-|w|^2}|w|^{2m}\,dA(w)$$
that
$$F(f)=\frac1{m!\pi}\inc f(w)\overline{g(w)}\,e^{-|w|^2}|w|^{2m}\,dA(w).$$
where
$$g(w)=\overline{F(K_m(\,\cdot\,,w))}.$$
It is clear that $g$ is entire and
\begin{eqnarray*}
|g(w)|&\le&\|F\|\|K_m(\,\cdot\,,w)\|_{p,m}\\
&=&\|F\|\left[\inc\left|\frac{e^{z\overline w}-p_m(z\overline w)}{(z\overline w)^m}\,
z^me^{-\frac12|z|^2}\right|^p\,dA(z)\right]^{\frac1p}\\
&=&\frac{\|F\|}{|w|^m}\left[\inc|e^{z\overline w}-p_m(z\overline w)|^p
e^{-\frac p2|z|^2}\,dA(z)\right]^{\frac1p}.
\end{eqnarray*}
According to Theorem~\ref{3}, there is a positive constant $C$ such that
$$|w^mg(w)|\le C\left[e^{\frac p2|w|^2}\right]^{\frac1p},\qquad w\in\C.$$
This shows that $g\in F^{\infty,m}$ and completes the proof of the theorem.
\end{proof}

\begin{thm}
Suppose $1\le p_0<p_1\le\infty$, $\theta\in(0,1)$, and
$$\frac1p=\frac{1-\theta}{p_0}+\frac\theta{p_1}.$$
Then $[F^{p_0,m},F^{p_1,m}]_\theta$, the complex interpolation space between $F^{p_0,m}$
and $F^{p_1,m}$, can be identified with $F^{p,m}$.
\label{15}
\end{thm}

\begin{proof}
That $[F^{p_0,m},F^{p_1,m}]_\theta\subset F^{p,m}$ follows from the definition of
complex interpolation, the fact that each $F^{p_k,m}$ is a closed subspace of
$L^{p_k}_m$, and the fact that $[L^{p_0}_m,L^{p_1}_m]_\theta=L^p_m$.

Conversely, if $f\in F^{p,m}\subset L^p_m$, then it follows from the fact that
$[L^{p_0}_m,L^{p_1}_m]_\theta=L^p_m$ there exists a function $F(z,\zeta)$, where 
$z\in\C$ and $0\le\re(\zeta)\le1$, and a positive constant $C$, such that
\begin{enumerate}
\item[(a)] $F(z,\theta)=f(z)$ for all $z\in\C$.
\item[(b)] $\|F(\,\cdot\,,\zeta)\|_{p_0,m}\le C$ for all $\re(\zeta)=0$.
\item[(c)] $\|F(\,\cdot\,,\zeta)\|_{p_1,m}\le C$ for all $\re(\zeta)=1$.
\end{enumerate}
Let $G(z,\zeta)=Q_mF(z,\zeta)$, where the projection $Q_m$ is applied with respect to
the variable $z$. Then for any fixed $\zeta$ the function $G(z,\zeta)$ is entire in
$z$, and the boundedness of the projection $Q_m$ on $L^{p_k}_m$ shows
that there is another positive constant $C$ such that
\begin{enumerate}
\item[(a)] $G(z,\theta)=f(z)$ for all $z\in\C$.
\item[(b)] $\|G(\,\cdot\,,\zeta)\|_{p_0,m}\le C$ for all $\re(\zeta)=0$.
\item[(c)] $\|G(\,\cdot\,,\zeta)\|_{p_0,m}\le C$ for all $\re(\zeta)=1$.
\end{enumerate}
This shows that $f\in[F^{p_0,m},F^{p_1,m}]_\theta$, and completes the proof of 
the theorem.
\end{proof}

\end{document}